\documentclass[12pt,a4paper]{article}
\usepackage[active]{srcltx}
\usepackage{amsfonts}
\usepackage{amsmath}
\usepackage{amsbsy}
\usepackage{amsxtra}
\usepackage{latexsym}
\usepackage{amssymb}
\usepackage{url}
\usepackage{cite}
\usepackage{pstricks,pst-node}
\makeatletter
\g@addto@macro\thesection.
\makeatother

\newtheorem{theorem}{Theorem}

\newtheorem{proposition}{Proposition}
\newtheorem{remark}{Remark}
\newtheorem{definition}{Definition}
\newtheorem{example}{Example}

\DeclareMathOperator{\intr}{int}

\DeclareMathOperator{\argmin}{argmin}

\newcommand{\lang}{\langle}
\newcommand{\rang}{\rangle}

\newcommand{\R}{\mathbb R}
\newcommand{\N}{\bf N}

\newcommand{\M}{\bf M}

\newcommand{\K}{\bf K}

\newcommand{\bv}{\bigvee}

\newenvironment{proof}{{\noindent\bf Proof.}}{\hfill$\Box$\\}

\begin{document}

\title{Characterization of the $\sigma$-Dedekind complete Riesz space
by the subadditivity of its positive part mapping}
\author{A. B. N\'emeth\\Faculty of Mathematics and Computer Science\\Babe\c s 
Bolyai University, Str. Kog\u alniceanu nr. 1-3\\RO-400084
Cluj-Napoca, Romania\\email: nemab@math.ubbcluj.ro}
\date{}

\maketitle

\begin{abstract} 

Two retractions $M$  and $N$ on convex cones $\M$ and
respectively $\N$ of a  real vector space $X$
are called mutually polar if $M+N=I$ and $MN=NM=0.$
In this note it is shown, that if the cones $\M$ and $\N$ are generating, 
$\sigma$-monotone complete, $M$ and $N$ are $\sigma$-monotone continuous,
then the subadditivity of $M$ and $N$ (with respect to the order relation induced by
$\M$ and respectively $\N$)
imply that $\M$ and $\N$ are lattice cones. $(X,\M)$ and  $(X,\N)$
become $\sigma$-Dedekind complete Riesz spaces, $M$ and $-N$ are the positive part ,
respectively the negative part mappings in $(X,\M)$; $N$ and $-M$ are the positive
part, respectively the negative part mappings in $(X,\N)$.

\end{abstract}

\section{Introduction}

If the Euclidean space $\R^m$ is ordered by a \emph{simplicial cone} , i. e.
a cone engendered by $m$ linearly independent elements of $\R^m$, then the ordering
is a \emph{latticial} one, i.e. for every pair $x,\,y\in \R^m$ there exists 
$x\vee y= \sup\{x,y\}$ and $x\wedge y=\inf\{x,y\}$, and the \emph{lattice operations}
$\vee$ and $\wedge$ are continuous. The famous theorem of Youdine \cite{Youdine1939}
asserts the converse of the above obvious assertion: If $\R^m$ is endowed
with a \emph{latticial ordering} $\leq$, with continuous lattice operations $\vee$ and $\wedge$, then the 
positive cone $\K=\{x\in \R^m: \,0\leq x\}$ of this ordering is a simplicial one.
The proof is far to be simple. This circumstance motivates the term \emph{Youdine cone}
for the simplicial cone in $\R^m$.

(The classification of all finite dimensional lattice cones  was done in \cite{Nemeth2004},
yielding as particular case the proof of Youdine's theorem too.)

An infinite dimensional generalization of Youdine's theorem is meaningless since
the geometry of closed, infinite dimensional lattice cones is fairly diverse.

Our approach is hence different: it is characterized a large class of Riesz spaces
by a single property: the subadditivity of their positive part mappings.

\section{Basic terminology}\label{bterm}

The nonempty set $\K$ in the real vector space $ X$ is called a \emph{convex cone} if 
\begin{enumerate}
	\item[(i)] $\lambda x\in \K$, for all $x\in \K$ and $\lambda \in \R_+=[0,\infty)$ and if 
	\item[(ii)] $x+y\in \K$, for all $x,y\in \K$. 
\end{enumerate}

A convex cone $\K$ is called \emph{pointed} if $\K\cap(-\K)=\{0\}$. 

A convex cone is called \emph{generating} if $\K-\K$$=X$. 

If $X$ is a topological vector space then a closed, pointed generating convex cone in it is called \emph{proper}.

The \emph{relation $\leq$ defined by the pointed convex cone $\K$} is given by $x\leq y$ if and only if 
$y-x\in \K$. Particularly, we have $\K=\{x\in X:0\leq x\}$. The relation $\leq$ is an \emph{order relation}, that 
is, it is reflexive, transitive and antisymmetric; it is \emph{translation invariant}, that is, $x+z\leq y+z$, 
$\forall x,y,z\in X$ with $x\leq y$; and it is \emph{scale invariant}, that is, $\lambda x\leq\lambda y$, 
$\forall x,y\in X$ with $x\leq y$ and $\lambda \in \R_+$.

Conversely, for every $\leq$ scale invariant, translation invariant and antisymmetric order relation in $X$ there exists a pointed 
convex cone $\K$, defined by $\K=\{x\in X:0\leq x\}$, and such that $x\leq y$ if and only if $y-x\in \K$. The cone $\K$ 
is called the \emph{positive cone of $(X,\leq)$}. $(X,\leq)$ (or $(X,\K)$) is called an \emph{ordered vector space, shortly o.v.s.}; we use
also the notation $\leq =\leq_K.$ 

The mapping $R:(X,\leq)\to (X,\leq)$ is said to be \emph{isotone} if $x\leq y\,\Rightarrow Rx\leq Ry.$
$R$ is called \emph{subadditive} if it holds $R(x+y)\leq Rx+Ry$, for any $x,\,y\in X$.

The ordered vector space $(X,\leq)$ is called \emph{latticially  ordered} if for every  $x,y\in X$ there exists 
$x\vee y:=\sup\{x,y\}$. In this case the positive cone $\K$ is called a \emph{lattice cone}.

\vspace{1mm}

\emph{The lattice cone is pointed and generating.}

\vspace{1mm}

A latticially ordered vector space is called a \emph{Riesz space}. 
Denote $x^+=0\vee x$ and $x^-=0\vee (-x)$. Then, $x=x^+-x^-$. $x^+$ is 
called the \emph{positive part} of $x$ and $x^-$ is called the \emph{negative part} of $x$. 
 The mapping 
$x\to x^+$ is called the \emph{positive part mapping}.

\begin{remark}\label{latt}

Let $(X,\K)$ be a Riesz space. Then the following properties of the mappings 
$^+ $ and $^-$ are immediate  (see e.g. 
\cite{LuxemburgZaanen1971},\,\cite{Schaefer1974}.)

1. The mapping $^+$ is \emph{idempotent}, that is, $(x^+)^+=x^+,\,\forall\,x\in X$.
 
2. The mapping $^+$ is \emph{subadditive}, that is $(x+y)^+\leq x^++y^+,\,\forall\,x,\,y\in X.$

3. If $x^+=(-x)^+ =0 \,\Rightarrow \,x=0$, since $\K$ is pointed.

4.$^+$ is \emph{isotone}, i.e. $x\leq y\, \Rightarrow \,x^+\leq y^+$;

5. $x^+-x^-=x,\,\forall\,x\in X$.

6.$ (x^+)^-=(x^-)^+=0, \,\,\forall\,x\in X.$

\end{remark}

If $X$ is endowed with a
vector space topology, then the continuity of the positive
part mapping is equivalent with the closedness of its positive cone $\K$.

The closed, pointed cone $\K$$ \subset X$ with $X$ a Banach space is called \emph{normal} 
if from $x_n\in \K$, $x_n\to 0$ and
$0\leq y_n\leq x_n$ it follows $y_n\to 0.$ The  ordered Banach space with normal positive cone is called normal.

The cone $\K$ and the  ordered Banach space $(X,\K)$ is called \emph{regular} if every monotone order bounded
sequence in $\K$  is convergent.
 
Every closed pointed convex cone in a finite dimensional Banach space is normal and regular.

Every closed normal cone in a separable Hilbert space is regular. (\cite{Mcarthur1970})


\section{Order complete ordered vector spaces}

\begin{definition}\label{oc}
The o.v.s. $(X,\leq)$ is called \emph{$\sigma$-monotone complete o.v.s.} if 
every increasing sequence $\{a_n\}_{n=1}^\infty$ in $X$
that is bounded from above has a supremum $\bv _{n=1}^\infty a_n \in X.$

The o.v.s. $(X,\leq)$ is called  \emph{Dedekind complete 
($\sigma$-Dedekind complete) o.v.s.} if 
every set $S$ (denumerable set $S$) $ \subset X$
that is bounded from above has a supremum in $X$.

A Riesz space is $\sigma$-monotone complete if and only if
it is $\sigma$-Dedekind complete.

\end{definition}

The $\sigma$-monotone completeness for an o.v.s. is a very mild condition.
The notion occurs in Riesz space theory and is related among others to the monotone
convergence theorems (see Example \ref{lebesgues}). The o.v.s.-s
in this context are Riesz spaces, but among our Examples \ref{regul}, \ref{hilbsig},
\ref{finsig}, there exists examples in some
operator algebras, where the $\sigma$-monotone complete ordering is not a 
latticial one (see e.g. \cite{Wright1972}).

Our main result is meaningful also in the
context of the ordered regular Banach spaces.

\begin{example}\label{regul}

 Every ordered regular Banach space $(X,\K)$ with a closed positive cone $\K$ is $\sigma$-monotone
complete o.v.s. Indeed, each increasing  order bounded sequence $(a_n)_{n=1}^\infty$ in $X$
is convergent to $a\in X$, $a_n\leq a, \,\,\forall \,n\in \mathbf N$ and $a\leq w$
whenever $a_n\leq w,\,\forall\, n\in \mathbf N$ since $\K$ is closed.

It is $\sigma$-Dedekind complete if and only if it is a Riesz space too. 
 
\end{example}

\begin{example}\label{hilbsig}
 Every closed normal cone $\K$ in a real separable Hilbert space $H$ is regular (\cite{Mcarthur1970})  , hence
$(H,\K)$ is {$\sigma$-monotone complete o.v.s..}

\end{example}

\begin{example}\label{finsig}

 Every closed convex  pointed cone $\K$ in the $m$-dimensional
real Euclidean space $\R^m$ is normal, hence regular and hence $(\R^m,\K)$ is
\emph{$\sigma$-monotone complete o.v.s.}.

\end{example}

\begin{example}\label{lebesgues}

 Let $(X,\mathcal A,\mu)$ be a measure space. Denote with $L=L(X,\mathcal A, \mu)$
the vector space of $\mu$-integrable real-valued functions on $X$. Under the order relation
$f\leq g\,\Leftrightarrow f(x)\leq g(x)\, \mu-$almost everywhere ($\mu$-a.e.) in $X$,
$L$ becomes a Riesz space $(L,\leq)$. The positive cone $\K$ of this space is the
set of $\mu$-a.e. non-negative $\mu$-integrable functions on $X$.

From the Lebesgue domination convergence theorem (see e.g. \cite{Bauer2001}) it follows then that if $\{f_n\}_{n=1}^\infty$
is an $\leq$-increasing sequence in $\K$ with an upper bound, then $f_n(x)\to f(x)$ \, $\mu$-a.e. in $X$,
$f\in \K$ is the supremum of the set $\{f_n\}_{n=1}^\infty$. Hence $(L,\leq)$ is a $\sigma$-order complete
o.v.s. 

\end{example}

\begin{example}\label{lex}

Consider the o.v.s. $(\R^2,\leq)$, where $\leq$ is the \emph{lexicographic ordering},
that is the ordering defined by

$$y= (y^1,y^2) \leq (x^1,x^2)=x,\,\Leftrightarrow\, x^1>y^1,\,\textrm{ or } x^1=y^1 \,\textrm{ and }\, y^2\leq x^2.$$

$(\R^2,\leq)$ is a Riesz space which is not $\sigma$-order complete.

Indeed $\{a_n\}_{n=}^\infty$ with $a_n=(0,n)$ is an increasing sequence with each  $w=(w^1,w^2),\, w^1>0$
an upper bound, but no such $w $ can be a least upper bound. 

\end{example} 


\section{$\sigma$-monotone continuous mappings}

\begin{definition} \label{simon}
 
By convention introduced in the preceding section,
the symbol $\bv _{n=1}^\infty a_n$ means that it is the supremum of the 
increasing sequence $\{a_n\}_{n=1}^\infty$ of the o.v.s. $(X,\leq)$.
Bearing in mind this convention the mapping $T:(X,\leq) \to (X,\leq)$ is called
\emph{$\sigma$-monotone continuous } if 
\begin{equation}\label{simon1}
 \bv_{n=1}^\infty Ta_n= T(\bv_{n=1}^\infty a_n)
\end{equation}
whenever $\bv_{n=1}^\infty a_n$ exists. The formula implicitly assumes that $T$
is isotone on the ordered set $\{a_n\}_{n=1}^\infty$ and the increasing sequence $\{Ta_n\}_{n=1}^\infty$
has its supremum $T(\bv_{n=1}^\infty a_n)$.

\end{definition}

\begin{example}\label{regsig}
Let $(X,\leq)$ be an ordered regular vector space with closed positive cone.
 Then every isotone continuous mapping $T:X\to X$ is $\sigma$-monotone continuous. 
\end{example}

\begin{example}\label{lattsig}

Let $(X.\K)$ be a Riesz space . Denote with $M$ the
positive part mapping of this space. We know that M (Remark \ref{latt})
is isotone and idempotent. Assume that the supremum $a=\bv _{n=1}^\infty a_n$ 
exists for the increasing sequence $\{a_n\}$. Then we have
\begin{equation}\label{pozresz}
\bv_{n=1}^\infty Ma_n= M(\bv_{n=1}^\infty a_n)=Ma.
\end{equation}
Indeed, the sequence $(Ma_n)_{n=1}^\infty$ is increasing and
the element $Ma$ is an upper bound for this sequence by the isotonicity of $M$.
Let $w$ be arbitrary with $Ma_n\leq w,\,\forall\, n$. 
Then by the isotone and idempotent property of 
$M$ and the fact that $Ma,\,w\in \K$, it follows that $Ma_n\leq Ma \leq Mw=w ,\,\forall \,n\in \mathbf N$.
Hence $Ma$ is the supremum of the increasing sequence $(Ma_n)_{n=1}^\infty$.
We have shown that \emph{the positive part mapping in the Riesz space is
$\sigma$-monotone continuous.}

\end{example}


\section{Mutually polar retractions on cones}

\begin{definition}\label{muture}
Let $X$ be a vector space, $0$ its zero mapping and $I$ its identity mapping. 
The mappings $M,\,N:\,X\to X$ are called
\emph{mutually polar retractions} if 
\begin{enumerate}
\item [(i)] $\M: =$ $MX$ and $\N:=$ $NX$ are convex cones, 
\item[(ii)] $M+N=I$,
\item[(iii)] $MN=NM=0$.
\end{enumerate}
Denote the order relation induced by $\M$ (respectively by $\N$) with $\leq_M$
(respectively by $\leq_N$). 
\end{definition}

{\bf Convention: Any order theoretic consideration on a retract is with respect 
to the order induced by its range.}

\vspace{2mm}

For instance $M$ is \emph{isotone} if $x\leq_My\, \Rightarrow\, Mx\leq_MMy$, $N$
 is \emph{subadditive} means $N(x+y)\leq_N Nx+Ny,\, \forall\,x,y \in X$.

 We will say that $M$ and $N$ 
are \emph{polar of each other}.

\begin{example}\label{positp}

Let $(X,\K)$ be a Riesz space Then bearing in mind the Remark \ref{latt}
the mappings $M$ and $N$ defined by $Mx=x^+,\,Nx=-x^-,\,\forall \,x\in X$, 
are mutually polar retractions.

We have seen (Example \ref{lattsig}) that $M$ is $\sigma$-monotone continuous.
The same conclusion holds also for the -negative part mapping denoted with $N$.

We mention also that $\M:=$ $MX=\K,$ and $\N:=$ $ NX=-\K$, hence booth $\M$ and $\N$ are generator cones.

We shall use for  simplicity the term \emph{lattice retraction} for the positive part mapping with respect the
order relation defined by a lattice  cone.

\end{example}

\begin{example}\label{proj}

Let $\K$ be a closed convex cone in the separable real
Hilbert space $(H,\lang,\rang)$.
Then $\K^\circ$ defined by
$$\K^\circ=\{x\in H:\, \lang y,x\rang\leq 0,\,\forall \, y\in \K\},$$
is called the \emph{polar of} $\K$. The polar is a closed convex cone.

Let $P_K:H\to \K$ and $P_{K^\circ}:H\to \K^\circ$ be the \emph{projection mappings} onto the closed convex cone $\K$
and $\K^\circ$ respectively  
, that is, the mappings defined by  
$$P_Kx=\argmin \{\|x-y\|:y\in \K\} \,\textrm{ and }\,P_{K^\circ}x=\argmin \{\|x-y\|:y\in\K^\circ\}.$$ 

Moreau's theorem \cite{Moreau1962} asserts that $P_K$ and $P_{K^\circ}$ are 
mutually polar retractions.

\end{example}

\begin{example}\label{egydim} 
Let $X$ be a Banach space $\K$ $\subset X$ a proper cone and $y\in \intr \K$. Denote by $\varphi$ the Minkowski
functional of $\K$ with respect to its interior point $y$.
Then the mappings defined by
$$Mx=\varphi(x)y,\,\,\textrm{ and }\,\,Nx=x-\varphi(x)y,\, \forall \,x\in X$$
are mutually polar retractions \cite{Nemeth2021}.
\end{example}

\section{Main results}

For the general mappings $Q,\,R,\,S: \,X\to X$ on the vector space $X$ we adopt the conventions: 
$$(Q+R)S=QS+RS,\,\textrm{ but }\, Q(R+S)=QS+RS \,\textrm{ only if }\,Q \,\textrm{ is additive }.$$

\begin{proposition}\label{idemp} 

If $M$ and $N$ are mutually polar retractions,then
\begin{enumerate}
\item[(i)] $M$ and $N$ are idempotent, that is, $M^2=M$ and $N^2=N$.
\item[(ii)] $\M=$ $\{x\in X:\,Nx=0\}$ and $\N=$ $\{x\in X:\,Mx=0\}$.
\item[(iii)] If $X$ is a topological vector space and one of retractions is continuous,
 then so is the other and the convex cones $\M$ and $\N$ are closed.
\end{enumerate}
\end{proposition}

\begin{proof}

(i) Since $M=(I-N)$ we have $M^2=(I-N)M=M-NM=M.$

The assertions (ii) and (iii) are immediate too.

\end{proof} 

\begin{proposition}\label{mn}
If for mutually polar subadditive retractions $M$ and $N$, their ranges
$\M$ and respectively $\N$ are both generating cones, then
$$\N=-\M.$$
\end{proposition}

\begin{proof}
We see first that
\begin{equation}\label{m}
\{Mx+My-M(x+y):\, x,\,y \in X\}=\M
\end{equation}
and
\begin{equation}\label{n}
\{Nx+Ny-N(x+y):\, x,\,y \in X\}=\N.
\end{equation}

Take $x\in \M$ arbitrarily. Since $\N$ is generating, there  are the elements $u,\,v\in \N$
such that $x=u-v$. Hence
$$x=Mx=Mx+Mv-M(x+v),$$
since $v,\,x+v\in \N$. This proves  (\ref{m}) and by symmetry also (\ref{n}).

We have finally
\begin{eqnarray}\label{eee}
\nonumber\{Nx+Ny-N(x+y):\, x,\,y \in X\}=\{(I-M)x+(I-M)y-(I-M)(x+y):\, x,\,y \in X\}=\\
-\{Mx+My-M(x+y):\, x,\,y \in X\},
\end{eqnarray}

\vspace{2mm}

which together (\ref{m}) and (\ref{n}) prove our claim.

\end{proof}

\begin{remark}\label{eeee} 
It is important to observe that in the case of mutually 
polar retractions $M$ and $N$ the condition $\N=-\M$, via  the formula (\ref{eee})
implies that the subadditivity of one of the retractions imply the subadditivity of the other.
\end{remark}

\begin{proposition}\label{adiso}
Let $M$  and $N$ be  mutually polar subadditive retractions, with generating cones
$\M$ and $\N$. Then they are isotone too.
\end{proposition}

\begin{proof}
We prove the isotonicity of $M$.

Take $v\in \N$ and $x\in X$ arbitrarily. 

Then
$$Mx\leq_M M(x-v)+Mv = M(x-v).$$
Hence 
$$Mx\leq_M My, \, \textrm{ with } y=x-v,\, \textrm{ or whenever } x-y \in \N.$$

Since by Proposition \ref{mn}, $\N=-\M$
 $$Mx\leq_MMy, \, \textrm{ whenever } x-y\in -\M , \, \textrm{ or whenever }\, x\leq_My.$$
\end{proof}

\begin{proposition}\label{is}
Suppose that $M$ and $N$ are mutually polar $\sigma$-order continuous retractions with
generating $\sigma$-order complete ranges $\M$ and $\N$ respectively. If $M$ is $\leq_M$-isotone,
($N$ is $\leq_N$-isotone) then $\M$ ($\N$) is a lattice cone.
\end{proposition}

\begin{proof} 

Suppose that $M$ is $\leq_M$-isotone. We put in this proof $\leq=\leq_M$ and suppose that all the order theoretic
 considerations in this proof are with respect to this relation.

 For the arbitrary fixed $y\in X$ we define $M_y:X\to X$ by $M_yx=y+M(x-y)$. Since $M$ is
	$\sigma$-order continuous and isotone, $M_y$ is also $\sigma$-order continuous and isotone. Moreover, $y\leq M_yx$.

We shall show that $\M$ is a lattice cone. That is , that for the arbitrary two elements $u$ and $v$
there exists $\sup \{u,v\}\in X$.
 If $u$ and $v$
	are comparable the statement is trivial. Suppose that they are not comparable. 
	First we remark that the set $\{u,v\}$ has an upper bound. Indeed, since $\M$ is generating, there
	exist $u_1,u_2,v_1,v_2\in \M$ such that $u=u_1-u_2$ and $v=v_1-v_2$. Hence, $u_1+v_1$ is an upper bound of
	the set $\{u,v\}$. Let $w$ be an arbitrary upper bound of the set $\{u,v\}$; that is, an arbitrary element
	of $(u+\M)\cap (v+\M)$. The mappings $M_u$ and $M_v$ are isotone. Moreover, $M_uw=u+M(w-u)=u+(w-u)=w$ and
	similarly $M_vw=w$. Consider the mappings $G=M_u\circ M_v$ 
	and $H=M_v\circ M_u$. They are isotone because $M_u$ and $M_v$ are. Moreover, $Gw=Hw=w$. 
	
	Denote $v_n=H^nv$, $u_1=M_uv$ and $u_n=G^{n-1}u_1$. We have $u\leq M_uv=u_1$. Also, $u\leq M_uv$ 
	implies $v\leq M_vu\leq M_v\circ M_uv=v_1$ and therefore $u_1=M_uv\leq M_uv_1$, or equivalently
	$u_1\leq M_u\circ M_v\circ M_uv=Gu_1=u_2$.
	Bearing in mind that $G,~H$ are isotone, 
	$Gw=Hw=w$, the relations
	\[v\leq v_1\leq\cdots\leq v_n\leq\cdots\leq w\] 
	and
	\[u\leq u_1\leq u_2\leq\cdots\leq u_n\leq\cdots\leq w\]
	can be verified by using mathematical induction. We further have
	\begin{eqnarray}
		\nonumber v_n=H^nv=(M_v\circ M_u)^nv=M_v\circ(M_u\circ M_v)^{n-1}\circ M_uv=
		M_v\circ G^{n-1}u_1\\
		=M_vu_n\label{vu}
	\end{eqnarray}
	and
	\begin{equation}\label{uv}
		u_{n+1}=Gu_n=M_u\circ M_vu_n=M_uv_n.     
	\end{equation}
	
	Since the sequences $\{u_n\}$ and $\{v_n\}$ are increasing and bounded above 
	by $w$, $u\leq u_n\leq w$ and $v\leq v_n\leq w$. From the $\sigma$-order completeness 
	of $\M$ and the $\sigma$-order continuity and isotonicity of $M_u$ and $M_v$ we have

	\begin{equation}\label{lim}
		u^*=\bv _{n=1}^\infty u_{n+1}=\bv_{n=1}^\infty M_uv_n=M_u(\bv _{n=1}^\infty v_n)=M_uv^*
		~\textrm{ and }
		\end{equation}
	\begin{equation}\label{lim1}
		v^*=\bv_{n=1}^\infty v_n=\bv_{n=1}^\infty M_vu_n=M_v(\bv _{n=1}^\infty u_n)=M_vu^*
	\end{equation}
	such that 
	\begin{equation}\label{leq}
		u\leq u^*\leq w~\textrm{ and }~v\leq v^*\leq w.
	\end{equation}
	Hence, we have $v^*-u^*+u^*-v=M(u^*-v)$ and $u^*-v^*+v^*-u=M(v^*-u)$, or 
	equivalently $(I-M)a=Na=c\in \N$ and $(I-M)b=Nb=-c$, where $a=u^*-v$, $b=v^*-u$ and $c=u^*-v^*$. Thus, we must 
	have $u^*-v^*=c=0.$ Hence $u^*=v^*=\sup \{u,v\}$.
\end{proof}

We remind the usage of the term \emph{lattice retraction} for the positive part mapping 
of the order relation induced by a lattice cone.

\begin{proposition}\label{lc}
Suppose that $M$ and $N$ are mutually polar retractions with the ranges $\M$
and respectively $\N$. If $\N=-\M$, $\M$ is a lattice  cone and $M$ is $\leq_M$-isotone,
then $M$ is the lattice retraction of $\M$.
\end{proposition}

\begin{proof}
Let $L$ be the lattice retraction of the lattice cone $M$. Then $L$ is $\leq_M$-isotone
and $x\leq_M Lx\in \M$ for any $x\in X$. Since $M$ is $\leq_M$-isotone we have
\begin{equation}\label{is1}
Mx\leq_MM(Lx)=Lx, \,\,\forall \,\, x\in X.
\end{equation}

For any $x\in X$ , 
$x-Mx=Nx\in \N=-\M$ and hence $x\leq_MMx$. By the $\leq_M$-isotonicity of $L$
$$Lx\leq_ML(Mx)=Mx,\, \forall \, x\in X,$$
which together with (\ref{is1}) proves the lemma.
\end{proof}

Gathering the results in the preceding propositions we conclude the following 

\begin{theorem}
 Let $M$ and $N$ be mutually polar retractions, let the cones $\M$ $=Mx$ and $\N$ $=NX$ be 
$\sigma$-monotone complete and generating convex cones, and let $M$ and $N$ be $\sigma$-monotone continuous.
Then the subadditivity of $M$ and $N$
imply that $\M$ and $\N$ are lattice cones, $(X,\M)$ and  $(X,\N)$
become $\sigma$-monotone complete Riesz spaces, $M$ and $-N$ are the positive part ,
respectively the negative part mappings in $(X,\M)$; $N$ and $-M$ are the positive
part, respectively the negative part mappings in $(X,\N)$.

\end{theorem}

\begin{proof}

We will work throughout in the proof under the assumptions: $M:X\to X$ and $N:X\to X$ are mutually polar retractions
with convex cone ranges $\M=$ $MX$ and $\N=$ $NX$. Then:

1. $M$ and $N$ are idempotent , $\M=$ $\{x\in X:\,Nx=0\}$, $\N=b$ $\{x\in X:\,Mx=0\}$ 
from Proposition \ref{idemp}.

2. If $M$ and $N$ are subadditive and $\M$ and $\N$ are generating, then $\N=-\M$ by 
Proposition \ref{mn}.

3. If $\N=-\M$ and $M$ and $N$ are subadditive, then they are isotone by
Proposition \ref{adiso}.

4. If $\M$ ($\N$) is $\sigma$-monotone complete and $M$ ($N$) is $\sigma$-monotone
continuous, and $\N=-\M$, then $\M$ ($\N$) is a lattice cone from Proposition \ref{is}.

5. If $\N=-\M$ and $\M$ ($\N$) is a lattice cone, then $M$ ($N$) is a lattice retraction
by Proposition \ref{lc}.

It also follows that $(X,\M)$ ($(X,\N)$) is a Riesz space with $M$ ($N$)
the positive part mapping, $-N$  ($-M$) the negative part mapping.

\end{proof}

Since the $\sigma$-monotone order continuity of the positive part mapping
and the minus negative part mapping 
in the general Riesz space (see Example \ref{lattsig} and Example \ref{positp})
, 
and the fact that the lattice cone is generating, we conclude that these
and the above theorem show that {\bf the only property of the subadditivity of the 
positive part mapping characterizes the $\sigma$-Dedekind complete Riesz space.}
(Compare also with Remark \ref{eeee}. ) 

\section{Remarks and comments}

1. The condition in the above theorem , that $\M$ and $\N$ are both generating cones (which is used
in Propositions \ref{mn}, \ref{adiso} and \ref{lc}) is essential.

Indeed, the mutually polar retractions $M$ and $N$ in the Example \ref{egydim} are both
subadditive (\cite{Nemeth2021}) , but $\M$ is one dimensional, hence if $\dim X\geq 2$ it is not generating,
hence not a lattice cone,
$\N$ is a proper cone, hence in general it is not a lattice cone.

2. In case of the metric projections $P_K$ and $P_{K^\circ}$ on the closed cone $\K$ and on its polar $\K^\circ$, of the
Hilbert space $H$, (Example \ref{proj}) the isotonicity of one of them implies the subadditivity of the other and converse.
(\cite{Nemeth2012}). In the case $H=\R^m$  booth $\K$ and $\K^\circ$ are simplicial cones of special forms.
In this case $P_K$ and $P_{K^\circ}$ are both subadditive if and only if $\K$ is
the positive orthant $\R^m_+$ of a Catesian reference system in $\R^m$. (They are in fact the lattice mappings on $\R^m_+$
and $-\R^m_+$ respectvely.)

3. The single place where $\sigma$ completeness and $\sigma$-order continuity is used 
is Proposition \ref{is}. The proof of this proposition follows and generalizes the proof
of S. Z. N\'emeth of his theorem in \cite{Nemeth2010}. (There was used for the  first time
in the literature the abstract notion of the retraction on a convex cone.)

4. The idea of the present note emerges from some investigations regarding the
so called \emph{asymmetric vector norms} in \cite{NemethNemeth2020}, \cite{Nemeth2020}
and \cite{Nemeth2021}.



\end{document}